\documentclass[11pt]{amsart}
\usepackage{latexsym,enumerate}
\usepackage{amssymb}
\usepackage[colorlinks]{hyperref}

\newtheorem{theorem}{Theorem}[section]
\newtheorem{lemma}[theorem]{Lemma}
\newtheorem{proposition}[theorem]{Proposition}

\newtheorem{MainTheorem}{Theorem}

\theoremstyle{definition}

\theoremstyle{remark}

\numberwithin{equation}{section}

\newcommand{\PGL}{{\mathrm {PGL}}}
\newcommand{\SL}{{\mathrm {SL}}}
\newcommand{\PSL}{{\mathrm {PSL}}}

\newcommand{\SU}{{\mathrm {SU}}}
\newcommand{\PSU}{{\mathrm {PSU}}}

\newcommand{\Sp}{{\mathrm {Sp}}}
\newcommand{\PSp}{{\mathrm {PSp}}}

\newcommand{\Ker}{\operatorname{Ker}}
\newcommand{\Aut}{{\mathrm {Aut}}}
\newcommand{\Out}{{\mathrm {Out}}}

\newcommand{\Irr}{{\mathrm {Irr}}}

\newcommand{\St}{{\mathsf {St}}}
\newcommand{\Stab}{{\mathrm {Stab}}}

\newcommand{\CC}{{\mathbb C}}

\newcommand{\AAA}{{\mathbb A}}

\newcommand{\FF}{{\mathbb F}}

\newcommand{\HC}{\mathcal{H}}

\newcommand{\OC}{\mathcal{O}}

\newcommand{\GC}{\mathcal{G}}

\newcommand{\ta}{\hspace{0.5mm}^{2}\hspace*{-0.2mm}}

\newcommand{\p}{^{\prime}}

\newcommand{\acd}{\mathsf{acd}}

\newcommand{\bC}{{\mathbf{C}}}
\newcommand{\bO}{{\mathbf{O}}}
\newcommand{\bN}{{\mathbf{N}}}
\newcommand{\bZ}{{\mathbf{Z}}}
\newcommand{\Al}{\textup{\textsf{A}}}
\newcommand{\Sy}{\textup{\textsf{S}}}

\begin{document}

\title[The average character degree and the It\^{o}-Michler theorem]
{The average character degree\\ and an improvement of the It\^{o}-Michler theorem}

\author{Nguyen Ngoc Hung}
\address{Department of Mathematics, The University of Akron, Akron,
OH 44325, USA} \email{hungnguyen@uakron.edu}

\author{Pham Huu Tiep}
\address{Department of Mathematics, Rutgers University, Piscataway, NJ 08854, USA}
\email{tiep@math.rutgers.edu}

\thanks{The second author gratefully acknowledges the support of the NSF (grant DMS-1840702).}
\thanks{The paper is partially based upon work supported by the NSF under grant DMS-1440140 while the authors were in residence at MSRI (Berkeley, CA), during the Spring 2018
semester. We thank the Institute for the hospitality and support. We
also thank Jay Taylor for an interesting discussion on the
extendibility property of unipotent characters of finite groups of
Lie type.}



\date{\today}

\begin{abstract} The classical
It\^{o}-Michler theorem states that the degree of every ordinary
irreducible character of a finite group $G$ is coprime to a prime
$p$ if and only if the Sylow $p$-subgroups of $G$ are abelian and
normal. In an earlier paper \cite{Hung-Tiep}, we used the notion of
average character degree to prove an improvement of this theorem
for the prime $p=2$. In this follow-up paper, we obtain a full
improvement for all primes.
\end{abstract}

\maketitle


\section{Introduction}

The classical It\^{o}-Michler theorem \cite{Ito,Michler} on
character degrees of finite groups asserts that the degree of every
ordinary irreducible character of a finite group $G$ is coprime to a
prime $p$ if and only if the Sylow $p$-subgroups of $G$ are abelian
and normal. Using the notion of the so-called \emph{average
character degree} introduced by Isaacs, Loukaki, and Moret\'{o} in
\cite{Isaacs-Loukaki-Moreto}, we proposed in \cite{Hung-Tiep} a new
direction, described below, to improve this theorem.

As usual, we use $\Irr(G)$ to denote the set of all ordinary
irreducible characters of $G$. Following \cite{Hung-Tiep}, let
\[
\Irr_p(G):=\{\chi\in\Irr(G)\mid \chi(1)=1 \text { or } p\mid
\chi(1)\}\] and \[\acd_p(G):=\frac{\sum_{\chi\in\Irr_p(G)}
\chi(1)}{|\Irr_p(G)|}
\]
so that $\acd_p(G)$ is the average degree of linear characters and
irreducible characters of $G$ with degree divisible by $p$. The
It\^{o}-Michler theorem is then equivalent to the statement that $\acd_p(G)=1$ if and
only if the Sylow $p$-subgroups of $G$ are abelian and normal.

We have observed that the normality of the Sylow $p$-subgroups of
$G$ can still be achieved when $\acd_p(G)$ is close to $1$. In
particular, we showed in \cite{Hung-Tiep} that if $\acd_2(G)<4/3$
then $G$ has a normal Sylow $2$-subgroup. On the other hand, the
Sylow $p$-subgroups are not abelian no matter how $\acd_p(G)$ is
close to $1$, as shown by the extraspecial $p$-groups.

The main result of this paper is a generalization of the
aforementioned result to all primes. For any given prime $p$, a key numerical invariant in this
result is the integer $\ell(p)$, which is defined to be the smallest positive integer $\ell$ such that
$\ell p+1$ is a prime power. Such an integer exists by Dirichlet's theorem.
Clearly, $\ell(2) = \ell(3) = \ell(7)=1$; moreover, $\ell(p) = 1$ if and only if
$p = 2$ or $p$ is a Mersenne prime. On the other hand, $\ell(5) = \ell(11) = \ell(13) = 2$, and
$\ell(17) = 6$.

\begin{MainTheorem}\label{theorem-main-1}
Let $p$ be a prime and let $G$ be a finite group with
$$\acd_p(G)< \frac{2\ell(p)p}{\ell(p)p+1}.$$
Then $G$ has a normal Sylow $p$-subgroup.
\end{MainTheorem}

The bound in Theorem~\ref{theorem-main-1} is best possible. Indeed, by the definition of $\ell(p)$, there is a prime
$r$ and an integer $m \geq 1$ such that $\ell(p)p = r^m-1$. Note that the cyclic group $C_{r^m-1}$ admits a faithful action on
the elementary abelian $r$-group $C^m_r$, leading to a semi-direct product $G = C_r^m \rtimes C_{r^m-1}$ with $\acd_p(G) = 2\ell(p)p/(\ell(p)+1)$, and
$G$ has a non-normal Sylow $p$-subgroup.


Of course if a finite group $G$ has a normal Sylow $p$-subgroup,
then $G$ is $p$-solvable. In fact, we have to establish
$p$-solvability before proving normality of Sylow $p$-subgroups.
The bound of $\acd_p(G)$ for $p$-solvability in the following
theorem is indeed best possible, shown by $\Al_5$ for $p=2,3$ and
$\PSL_2(p)$ for $p\geq 5$. Recall that $\bO^{p'}(G)$ is the minimal
normal subgroup of $G$ whose quotient is a $p'$-group.

\begin{MainTheorem}\label{theorem-main-3}
Let $p$ be a prime and set $a_2:=5/2$, $a_3:=7/3$ and $a_p:=(p+1)/2$
if $p\geq 5$. Let $G$ be a finite group such that $\acd_p(G)<a_p$.
Then $\bO^{p'}(G)$ is solvable and, in particular, $G$ is
$p$-solvable.
\end{MainTheorem}

We need to use the classification of finite simple groups to prove
Theorem~\ref{theorem-main-3}. In particular, the classification is
used to show the existence of an extendible irreducible character of
degree divisible by $p$ in non-abelian simple groups (see
Theorem~\ref{theorem-simple-groups-key1}). We believe that this
extendibility-divisibility result will be useful in other purposes
as well.

A similar statement to Theorem~\ref{theorem-main-3} may
still be true if we restrict our attention to only real-valued
characters or even strongly real characters, and this would
significantly improve the results in \cite{Marinelli-Tiep,Tiep}.
However, to prove it, one would first need to prove a real, respectively strongly
real, version of Theorem~\ref{theorem-simple-groups-key1}, which seems
very difficult to prove at the moment.

In view of Theorem~\ref{theorem-main-1},
it is reasonable to conjecture that the index $[G:\bN_G(P)|$ is
bounded in terms of $\acd_p(G)$, where $P$ is a Syllow $p$-subgroup
of $G$. Even the weaker statement that the number of non-abelian
composition factors of $G$ of order divisible by $p$ is bounded in
terms of $\acd_p(G)$ seems highly nontrivial to prove.


\section{Preliminaries}

Throughout the paper, let $n_d(G)$ denote the number of irreducible characters of degree
$d$ of a finite group $G$. The following lemma controls the number
$n_1(G)$ of linear characters in a special situation.

\begin{lemma}\label{lemma-n1-n2}
Let $G$ be a finite group with a non-abelian minimal normal subgroup
$N$. Assume that there is some $\mu\in\Irr(N)$ such that $\mu$ is
extendible to the inertia subgroup $I_G(\mu)$. Then $n_1(G)\leq
n_d(G)[G:I_G(\mu)]$ where
$d:=\mu(1)[G:I_G(\mu)]$. 
\end{lemma}

\begin{proof}
This is Proposition~2.3(i) of \cite{Hung-Tiep}.
\end{proof}

The following lemma allows us to focus on special subsets of $\Irr_p(G)$ in a number of situations:

\begin{lemma}\label{L-subset}
Let $G$ be a finite group, $p$ be a prime and let $\Irr^*(G)$ be a subset of $\Irr_p(G)$ that contains all linear characters of $G$.
Suppose that $\acd_p(G) \leq p$. Then
$$\acd_p(G) \geq \frac{\sum_{\chi \in \Irr^*(G)}\chi(1)}{|\Irr^*(G)|}.$$
In particular, if $N \lhd G$, $N \leq G'$, and $\acd_p(G) \leq p$, then $\acd_p(G/N) \leq \acd_p(G)$.
\end{lemma}

\begin{proof}
Let
$$A:=\sum_{\chi \in \Irr_p(G)}\chi(1),~B:=|\Irr_p(G)|,~C:= \sum_{\chi \in \Irr_p(G) \smallsetminus \Irr^*(G)}\chi(1),~D:=|\Irr_p(G) \smallsetminus \Irr^*(G)|.$$
By assumption, $A/B = \acd_p(G) \leq p$ and so $A \leq pB$. On the other hand, as $\Irr^*(G)$ contains all linear characters of $G$,
$C \geq pD$. It follows that
$$\acd_p(G) = A/B \geq \frac{A-B}{C-D} = \frac{\sum_{\chi \in \Irr^*(G)}\chi(1)}{|\Irr^*(G)|},$$
proving the first statement. For the second statement, note that the subset $\Irr_p(G/N)$ of $\Irr_p(G)$ contains all linear characters of $G$.
\end{proof}

The next lemma will be used frequently. It is well known, but we
include a proof for completeness.

\begin{lemma}\label{lemma-extension} Let $S$ be a non-abelian simple group and
let $N:=S\times\cdots\times S$, a direct product of $k$ copies of
$S$. Suppose that $\lambda\in\Irr(S)$ is extendible to
$I_{\Aut(S)}(\lambda)$. Then
$\chi:=\lambda\times\cdots\times\lambda\in\Irr(N)$ is extendible to
$I_{\Aut(N)}(\chi)$.
\end{lemma}

\begin{proof} Let
$\mathrm{Orb}(\lambda)$ denote the orbit of $\lambda$ in the action
of $\Aut(S)$ on $\Irr(S)$. As $\Aut(N)$ acts transitively on the
simple direct factors of $N$, the orbit of $\chi$ under the action
of $\Aut(N)$ is
\[
\mathrm{Orb}(\chi):=\{\lambda_1\times\lambda_2\times\cdots\times\lambda_k
\in\Irr(N)\mid \lambda_i\in \mathrm{Orb}(\lambda)\}.
\]
By assumption, $\chi$ is invariant under $I_{\Aut(S)}(\lambda)\wr
\Sy_k$. On the other hand, as $\Aut(N)=\Aut(S)\wr \Sy_k$, we have
\[
[\Aut(N):I_{\Aut(S)}(\lambda)\wr
\Sy_k|=[\Aut(S):I_{\Aut(S)}(\lambda)]^n=|\mathrm{Orb}(\lambda)|^n=|\mathrm{Orb}(\chi)|.
\]
Therefore we conclude that
$I_{\Aut(N)}(\chi)=I_{\Aut(S)}(\lambda)\wr \Sy_n$.

Let $\widetilde{\lambda}\in\Irr(I_{\Aut(S)}(\lambda))$ be an
extension of $\lambda$. Suppose that $\widetilde{\lambda}$ is
afforded by a $\CC I_{\Aut(S)}(\lambda)$-module $V$. Then
$I_{\Aut(N)}(\chi)=I_{\Aut(S)}(\lambda)\wr \Sy_k$ acts naturally on
$V^{\otimes k}$ and it follows that the character afforded by the
$\CC I_{\Aut(N)}(\chi)$-module $V^{\otimes k}$ is an extension of
$\chi$, which means that $\chi$ is extendible to $I_{\Aut(N)}(\chi)$.
\end{proof}

The next proposition will be critical in the proof of $p$-solvability of
finite groups with small $\acd_p$. The proof follows the idea of
\cite[Theorem 3.5]{Tiep}.

\begin{proposition}\label{pro-G/K-divisible-by-p}
Let $G$ be a finite group with a minimal normal subgroup $N\cong
S_1\times \cdots\times S_n$, where $n\geq 2$ and the $S_i$'s are all
isomorphic to a non-abelian simple group $S$. Let $K$ be the kernel
of the action of $G$ on $\{S_1,\ldots,S_n\}$. If $|G/K|$ is divisible
by a prime $p$, then there exists $\mu\in\Irr(N)$ such that
$\mu(1)\geq 12$, $\mu$ is extendible to a character of $I_G(\mu)$,
and $[G:I_G(\mu)]$ is divisible by $p$.
\end{proposition}

\begin{proof} Consider
the faithful action of $G/K$ on the set $\{S_1,S_2,\ldots,S_n\}$. As
$p||G/K|$, by \cite[Lemma 8]{Casolo-Dolfi}, there are two disjoint
subsets $\Omega_1$ and $\Omega_2$ of $\{S_1,S_2,\ldots,S_n\}$ such that
\begin{equation}\label{index1}
  p\mid [(G/K):\Stab_{G/K}(\Omega_1,\Omega_2)]=[G:\Stab_G(\Omega_1,\Omega_2)].
\end{equation}
Without loss we may assume that $\bC_G(N)=1$. Set
$$N_1 := \prod_{S_i\in \Omega_1} S_i,~N_2 := \prod_{S_i\in \Omega_2} S_i.$$
By \cite[Theorem 3.2]{Tiep}, we can find two irreducible characters,
say $\alpha$ and $\beta$, of $S$ of distinct degrees $\geq 3$, each of which
is extendible to its inertia subgroup in $\Aut(S)$. Consider the
irreducible characters
\[\tilde\alpha:=\alpha\times \cdots \times\alpha\in\Irr(N_1),~\tilde\beta:=\beta\times\cdots\times\beta\in\Irr(N_2).\]
Using Lemma~\ref{lemma-extension}, we deduce that
$\tilde\alpha$ and $\tilde\beta$ are extendible to their
respective inertia subgroups $I_{\Aut(N_1)}(\tilde\alpha)$ and $I_{\Aut(N_2)}(\tilde\beta)$.
It then follows that $\tilde\alpha\otimes \tilde\beta\in \Irr(M)$ is extendible to its inertia subgroup
in $\Aut(N_1) \times \Aut(N_2)$.
Let
$$N_3:=\prod_{S_i \notin \Omega_1\cup \Omega_2}S_i,$$
and consider
\[\mu:=\tilde\alpha\times \tilde\beta\times 1_{N_3}\in \Irr(N).\]
Since $\alpha$ and $\beta$ are non-principal and have distinct degrees, we see that
$$N, N_1N_2 \lhd J \leq \Stab_G(\Omega_1,\Omega_2) = \bN_G(N_1) \cap \bN_G(N_2)$$
for $J := I_G(\mu)$. Observe that
$$N_3 \lhd \bC_J(N_1N_2),~~N_1N_2 \cap \bC_J(N_1N_2) = 1.$$
Hence we can extend $\mu$ canonically to the character $\tilde\mu$ of
$$N\bC_J(N_1N_2) = N_1N_2 \times \bC_J(N_1N_2)$$
that is trivial on $\bC_J(N_1N_2)$ and can be viewed as
the character $\tilde\alpha \otimes \tilde\beta$ of
$$N\bC_J(N_1N_2)/\bC_J(N_1N_2) \cong N_1 \times N_2.$$
Since $N\bC_J(N_1N_2) \lhd J$ and $J/\bC_J(N_1N_2)$ embeds in
$\Aut(N_1) \times \Aut(N_2)$, the aforementioned extendibility of
$\tilde\alpha \otimes \tilde\beta$ implies that $\tilde\mu$ extends to $J = I_G(\mu)$.

Finally, since $J \leq \Stab_G(\Omega_1,\Omega_2)$, we have $p \mid
[G:J]$ by \eqref{index1}. Also, as $\alpha$ and $\beta$ have  distinct degrees at
least $3$, we have $\mu(1)\geq 12$, and the proof is complete.
\end{proof}


\section{$p$-Solvability}

As mentioned in the introduction, we need the classification of finite simple groups to
prove the $p$-solvability. Indeed, the classification is needed in
the next two theorems, whose proofs are deferred to
Section~\ref{section-last}.

\begin{theorem}\label{theorem-simple-groups-key1}
Let $p$ be a prime and let $S$ be a non-abelian simple group of
order divisible by $p$. Then there exists $\theta\in\Irr(S)$ such
that $p\mid \theta(1)$ and $\theta$ is extendible to a character of
$I_{\Aut(S)}(\theta)$.
\end{theorem}

\begin{theorem}\label{theorem-simple-groups-key2}
Let $p$ be a prime and let $S$ be a non-abelian simple group of
order not divisible by $p$. Then there exists a non-principal
character $\theta\in\Irr(S)$ such that $\theta$ is extendible to a
character of $I_{\Aut(S)}(\theta)$ and $p\nmid
|I_{\Aut(S)}(\theta)|$.
\end{theorem}

With Theorem~\ref{theorem-simple-groups-key1} in hand, we obtain the
following.

\begin{theorem}\label{theorem-extendible-characters-G}
Let $G$ be a finite group with a non-abelian minimal normal subgroup
$N$ of order divisible by a prime $p$. Then there exists
$\mu\in\Irr(N)$ such that $p\mid\mu(1)$ and $\mu$ is extendible to
$I_G(\mu)$.
\end{theorem}

\begin{proof}
Suppose that $N\cong S\times S\times\cdots\times S$, a direct
product of $r$ copies of a non-abelian simple group $S$ with $p\mid
|S|$. With no loss of generality, we may assume that $\bC_G(N)=1$.
Then $N\unlhd G \leq \Aut(N)=\Aut(S)\wr \Sy_r.$

By Theorem~\ref{theorem-simple-groups-key1}, we can find a character
$\theta\in\Irr(S)$ such that $p\mid \theta(1)$ and $\theta$ is
extendible to $I_{\Aut(S)}(\theta)$. Let
$\mu:=\theta\times\cdots\times\theta\in\Irr(N)$. Now using
Lemma~\ref{lemma-extension}, we have that $\mu$ is extendible to a
character of $I_{\Aut(N)}(\mu)$, which implies that $\mu$ is
extendible to a character of $I_G(\mu)=G\cap I_{\Aut(N)}(\mu)$.
Finally we note that $p\mid \mu(1)$ since $p\mid \theta(1)$.
\end{proof}

\begin{theorem}\label{theorem-main-2-again-1}
Let $p$ be an odd prime and set $a_p:=7/3$ if $p=3$ and
$a_p:=(p+1)/2$ if $p\geq 5$. Let $G$ be a finite group such that
$\acd_p(G)<a_p$. Then $G$ is $p$-solvable.
\end{theorem}

\begin{proof}
(i) Assume that the statement is false and let $G$ be a minimal
counterexample.  In particular, $G$ is not $p$-solvable. Take a
minimal normal subgroup $N$ of $G$ such that $N\leq G'$.
By Lemma \ref{L-subset}, $\acd_{p}(G/N)\leq \acd_{p}(G)<a_p$. The minimality of
$G$ then implies that $G/N$ is $p$-solvable, which in turns implies
that $N$ is not $p$-solvable since $G$ is not. This means that $N$
is a direct product of $r$ copies of a non-abelian simple group, say $S$,
with $p\mid |S|$.

Applying Theorem~\ref{theorem-extendible-characters-G}, we can now
find a $\mu\in\Irr(N)$ such that $p\mid\mu(1)$ and $\mu$ is
extendible to a character of $I_G(\mu)$, and the latter induced to $G$ yields an irreducible
character of $G$ of degree $d:=\mu(1)[G:I_G(\mu)]$. It then follows from
Lemma~\ref{lemma-n1-n2} that
\[
n_{1}(G)\leq n_{d}(G)[G:I_G(\mu)].
\]
Clearly, $(p-1)[G:I_G(\mu)]\leq 2p[G:I_G(\mu)]-p-1$. Therefore,
\[ (p-1)n_{1}(G)\leq (p-1)n_{d}(G)[G:I_G(\mu)]\leq
n_{d}(G)(2p[G:I_G(\mu)]-p-1).
\]
Since $\mu(1)\geq p$, we then have $p[G:I_G(\mu)] \leq d$, and so
\[
(p-1)n_{1}(G) \leq (2d-p-1)n_{d}(G).
\]
Now let $\Irr^*(G) := \{ \chi \in \Irr(G) \mid \chi(1) =1 \mbox{ or }d\}$. Since $\acd_p(G) \leq p$, by Lemma \ref{L-subset} we have
\[\acd_{p}(G) \geq \frac{\sum_{\chi \in \Irr^*(G)}\chi(1)}{|\Irr^*(G)|} = \frac{n_1(G) + dn_d(G)}{n_1(G)+n_d(G)}
 \geq \frac{p+1}{2}.\]
This is a contradiction when $p\geq 5$.

\smallskip
(ii) It now remains to consider the case $p=3$. Note that $\Al_5$ and
$\PSL_2(7)$ are the only non-abelian simple groups with an
irreducible character of degree $3$. If $S=\PSL_2(7)$, then it has an irreducible
character of degree 6 that is extendible to a character of
$\Aut(\PSL_2(7))$. Therefore, if $N \cong S^r$ is not
isomorphic to $\Al_5$, then from the construction of
$\mu$ in the proof of Theorem~\ref{theorem-extendible-characters-G},
we have $\mu(1)=\theta(1)^r\geq 6$.
We can now repeat the arguments in (i)  to get $\acd_{3}(G)\geq 7/2 > a_3$,
and this violates the hypothesis.

\smallskip
(iii) The only case left is $p=3$ and $N\cong \Al_5$. We then observe that
$N$ has two distinct irreducible characters, say $\mu_1$ and
$\mu_2$, of degree 3 that are fused under $\Sy_5\cong\Aut(N)$. Here
$I_{\Aut(N)}(\mu_1)=I_{\Aut(N)}(\mu_2)=N$ and
$I_G(\mu_1)=I_G(\mu_2)$ has index 1 or 2 in $G$. Note
that $\bC_G(N) \cap N = 1$, $N\bC_G(N) \leq I_G(\mu_1)$, and $G/N\bC_G(N) \hookrightarrow \Out(N) = C_2$.
Suppose $[G:I_G(\mu_1)]=2$.  Then $I_G(\mu_1) = N \times \bC_G(N)$ and so $\mu_1$ extends to $I_G(\mu_1)$.
It follows that $G$ has an irreducible character of degree $6$ lying above $\mu_1$, and $n_{1}(G)\leq 2n_{6}(G)$ by Lemma
\ref{lemma-n1-n2}. Arguing as above, we obtain
easily that $\acd_{3}(G)\geq 8/3$. On the other hand, if
$[G:I_G(\mu_1)]=1$ then $G$ does not induce an outer automorphism of
$N=\Al_5$, and so $G\cong N\times \bC_G(N)$. In this case, the outer tensor product of each $\mu_{1,2}$ with a linear
character of $\bC_G(N)$ yields an irreducible character of degree $3$ of $G$, hence
$$n_1(G) = n_1(\bC_G(N)),~n_3(G) \geq 2n_1(\bC_G(N)),$$
and so $n_1(G) \leq n_3(G)/2$. It now follows that
$\acd_{3}(G)\geq 7/3=a_3$, a contradiction.
\end{proof}

We are now ready to prove Theorem~\ref{theorem-main-3},
which is restated below.

\begin{theorem}\label{theorem-main-2-again-2}
Let $p$ be a prime and set $a_2:=5/2$, $a_3:=7/3$ and $a_p:=(p+1)/2$
if $p\geq 5$. Let $G$ be a finite group such that $\acd_p(G)<a_p$.
Then $\bO^{p'}(G)$ is solvable.
\end{theorem}

\begin{proof}
The case $p=2$ was already proved in \cite[Theorem 1.2]{Hung-Tiep}.
Therefore we may assume that $p$ is odd.

Let $G$ be a minimal counterexample to the statement. In particular,
$\bO^{p'}(G)$ is not solvable. Let $N/N_1$ be a non-abelian chief
factor of $G$ inside $\bO^{p'}(G)$ such that $|N|$ is smallest
possible. Then clearly $N$ must be perfect, so that
$N_1\lhd N=N' \lhd G'$. It follows by Lemma \ref{L-subset} that
\[
\acd_{p}(G/N_1)\leq \acd_{p}(G)<a_p,
\]
and thus, by the minimality of $G$, we have $N_1=1$. This means that
$N$ is a non-abelian minimal normal subgroup of $G$. Suppose that
$$N= S_1\times S_2\times\cdots\times S_n,$$
where the $S_i$'s are all isomorphic to a non-abelian simple group $S$.

By Theorem~\ref{theorem-main-2-again-1} we know that $G$ is
$p$-solvable. So $S$ is a $p'$-group. Since $N$ is non-abelian, we
have $N\nleq \bC_G(N)$, which implies that
$\bO^{p'}(G)\nleq \bC_G(N)$. It then follows that
$|G/\bC_G(N)|$ is divisible by $p$.

Now let $K$ be the kernel of the transitive action of $G$ on the set
$\{S_1,S_2,\ldots,S_n\}$ of the $n$ simple direct factors of $N$. Since
$\bC_G(N)\leq K$, we
deduce that
\[
\text{either } p\mid |K/\bC_G(N)| \text{ or } p\mid |G/K|.
\]

Let us first consider the case $p\mid |G/K|$. Then $n\geq 2$ and, by
Proposition~\ref{pro-G/K-divisible-by-p}, we can find $\mu\in\Irr(N)$
such that $\mu(1)\geq 12$, $\mu$ is extendible to a character of
$I_G(\mu)$, and $[G:I_G(\mu)]$ is divisible by $p$.
Lemma~\ref{lemma-n1-n2} then implies that $n_{1}(G)\leq
n_{d}(G)[G:I_G(\mu)]$, where $d:=\mu(1)[G:I_G(\mu)]\geq
12[G:I_G(\mu)]$ is divisible by $p$. Now we can proceed as in the
proof of Theorem~\ref{theorem-main-2-again-1} to show that
$\acd_{p}(G)\geq a_p$, and this is a contradiction.

We now may assume that $p\mid |K/\bC_G(N)|$. Recall that $S$ is a
$p'$-group. Therefore, by Theorem~\ref{theorem-simple-groups-key2},
we can find $\theta\in\Irr(S)$ such that $\theta$ is extendible to a
character of $I_{\Aut(S)}(\theta)$ and $p\nmid
|I_{\Aut(S)}(\theta)|$. Arguing as in the proof of
Theorem~\ref{theorem-extendible-characters-G} and viewing $N$ as a
subgroup of $G/\bC_G(N)$, we then can show that
$\mu:=\theta\times\cdots\times\theta\in\Irr(N)$ is extendible to a
character of \[I_{G/\bC_G(N)}(\mu)=(G/\bC_G(N))\cap
I_{\Aut(N)}(\mu),\] where $I_{\Aut(N)}(\mu)=I_{\Aut(S)}(\theta)\wr
\Sy_n$. In particular, $\mu$ extends to a character of $I_G(\mu)$.

We claim that $[G:I_G(\mu)]$ is divisible by $p$. Recall that $K$ is
the kernel of the action of $G$ on $\{S_1,S_2,\ldots,S_n\}$. Hence
$K/\bC_G(N)$ is contained in $\Aut(S)^n$, the base subgroup of the
wreath product $\Aut(N)=\Aut(S)\wr\Sy_n$. It follows that the index
\[
\left[\frac{K}{\bC_G(N)}:\left(\frac{K}{\bC_G(N)}\cap
I_{\Aut(N)}(\mu)\right)\right]=\left[\frac{K}{\bC_G(N)}:\left(\frac{K}{\bC_G(N)}\cap
I_{\Aut(S)}(\theta)^n\right)\right]
\]
is divisible by $p$ since $p\mid |K/\bC_G(N)|$ and $p\nmid
|I_{\Aut(S)}(\theta)|$. In particular, the index
\begin{align*}
[G:I_G(\mu)]=&[(G/\bC_G(N)):I_{G/\bC_G(N)}(\mu)]\\
=&\left[\frac{G}{\bC_G(N)}: \left(\frac{G}{\bC_G(N)} \cap
I_{\Aut(N)}(\mu)\right)\right]
\end{align*}
is also divisible by $p$, as claimed above.

Now we again apply Lemma~\ref{lemma-n1-n2} to have $n_{1}(G)\leq
n_{d}(G)[G:I_G(\mu)],$ where $d:=\mu(1)[G:I_G(\mu)]$ is divisible by
$p$. When $p\geq 5$ we can argue as in p. (i) of the proof of
Theorem~\ref{theorem-main-2-again-1} to show that $\acd_{p}(G)\geq
a_p$. Suppose $p=3$. Then, as $N$ is a $p'$-group, $N$ must be a simple
Suzuki group. Hence, we can choose $\mu$ so that $\mu(1)\geq
6$, which implies that $\acd_{3}(G)>7/3=a_3$, again a contradiction.
\end{proof}


\section{Normality of Sylow $p$-subgroups}

We start with a lemma.

\begin{lemma}\label{lemma-orbit}
Let $p$ be a prime and let $G=N\rtimes H$ where $N$ is an abelian
group. Assume that $\acd_p(G)\leq p$. Then, in the action of $H$ on
$\Irr(N)\smallsetminus \{1_N\}$, there exists an orbit $\mathcal{O}$ such that
$|\mathcal{O}|=1$ or $p\mid |\mathcal{O}|$ and \[
\frac{|\mathcal{O}|(f+1)}{|\mathcal{O}|+f}\leq \acd_p(G),
\] where $f$ is the number of $H$-orbits on $\Irr(N) \smallsetminus \{1_N\}$ whose sizes are 1 or divisible
by $p$.
\end{lemma}

\begin{proof}
(i) Let
$\{1_N=\alpha_0,\alpha_1,\ldots,\alpha_f,\alpha_{f+1}\ldots,\alpha_g\}$
be a set of representatives of the $H$-orbits on
$\Irr(N)$, where $\{\alpha_0,\alpha_1,\ldots,\alpha_f\}$ are
representatives of those orbits whose sizes are 1 or divisible by
$p$. For each $0\leq i\leq g$, let $I_i$ be the inertia subgroup of
$\alpha_i$ in $G$, and set
$$n_{i,p}:=\sum_{p|k} n_{k}(I_i/N),~~s_{i,p}:=\sum_{\lambda\in\Irr(I_i/N),~p|\lambda(1)} \lambda(1).$$

Since $G$ splits over $N$, every $I_i$ also splits over $N$, and in fact $I_i = N/\Ker(\alpha_i) \times I_H(\alpha_i)$
as $N$ is abelian. It follows that $\alpha_i$ extends to a unique linear character
$\beta_i$ of $I_i$ that is trivial at $I_H(\alpha_i)$. Gallagher's theorem
then provides us with a bijective mapping $\lambda\mapsto \lambda\beta_i$
from $\Irr(I_i/N)$ to the set of irreducible characters of $I_i$
lying above $\alpha_i$. By the Clifford correspondence, we then
obtain a bijection $\lambda\mapsto (\lambda\beta_i)^G$ from
$\Irr(I_i/N)$ to the set of irreducible characters of $G$ lying
above $\alpha_i$. We note that
$(\lambda\beta_i)^G(1)=[G:I_i]\lambda(1)$, and hence $p\mid
(\lambda\beta_i)^G(1)$ if and only if either $p\mid [G:I_i]$ or
$p\mid \lambda(1)$.

\smallskip
(ii) Now we take $\Irr^*(G)$ to be the set of irreducible characters $(\lambda\beta_i)^G$, where $\lambda \in \Irr(I_i/N)$ and exactly one of the following
holds:
\begin{enumerate}[\rm(a)]
\item $I_i=G$, $\lambda(1) = 1$;
\item $p \mid \lambda(1)$;
\item  $p$ divides $[G:I_i]$ and $\lambda(1) = 1$.
\end{enumerate}
By its construction and the above discussion, $\Irr^*(G)$ is contained in $\Irr_p(G)$ and contains all linear characters of $G$. As
$\acd_p(G) \leq p$, Lemma \ref{L-subset} applies to $\Irr^*(G)$. Note
that the number of linear characters in $\Irr^*(G)$ is $\sum_{i:\,G=I_i}n_1(I_i/N)$ (as they all come from type (a).
On the other hand, the non-linear characters in $\Irr^*(G)$ come from types (b) and (c), and so
the number of them is $\sum_{i:\,p|[G:I_i]}n_1(I_i/N)+\sum_{i=0}^g
n_{i,p}$. It follows from Lemma \ref{L-subset} that
\[
\sum_{\chi\in\Irr^*(G)}\chi(1) \leq \acd_p(G)\left(\sum_{i=0}^f n_1(I_i/N)+\sum_{i=0}^g
n_{i,p}\right).
\]
On the other hand,
\[
\sum_{\chi\in\Irr^*(G)}\chi(1) = \sum_{i=0}^f [G:I_i]n_1(I_i/N)+\sum_{i=0}^g [G:I_i]s_{i,p}\]
Since $s_{i,p}\geq pn_{i,p}$, it follows that
\[
\sum_{\chi\in\Irr^*(G)}
\chi(1)\geq\sum_{i=0}^f [G:I_i]n_1(I_i/N)+p\sum_{i=0}^g
[G:I_i]n_{i,p}
\]
Therefore, we obtain
\[
\acd_p(G)\left(\sum_{i=0}^f n_1(I_i/N)+\sum_{i=0}^g
n_{i,p}\right)\geq \sum_{i=0}^f [G:I_i]n_1(I_i/N)+p\sum_{i=0}^g
[G:I_i]n_{i,p}
\]
Using the hypothesis $\acd_p(G)\leq p$, we deduce that
\[
\acd_p(G)\sum_{i=0}^f n_1(I_i/N)\geq \sum_{i=0}^f [G:I_i]n_1(I_i/N),
\]
which is equivalent to
\begin{equation}\label{est1}
(\acd_p(G)-1)n_1(G/N)\geq \sum_{i=1}^f ([G:I_i]-\acd_p(G))n_1(I_i/N)
\end{equation}
since $I_0=G$.

Now, observe that $n_1(G/N)=[(G/N):(G/N)'|$ and
$n_1(I_i/N)=[(I_i/N):(I_i/N)'|$. Therefore $n_1(G/N)\leq
[G:I_i]n_1(I_i/N)$ for every $1\leq i\leq f$. Thus
\[
n_1(G/N)\leq \frac{1}{f} \sum_{i=1}^f [G:I_i]n_1(I_i/N).
\]
Together with \eqref{est1}, we deduce that
\[
\frac{\acd_p(G)-1}{f}\sum_{i=1}^f [G:I_i]n_1(I_i/N)\geq \sum_{i=1}^f
([G:I_j]-\acd_p(G))n_1(I_i/N).
\]
Therefore, there must exist some index $1\leq j\leq f$ so that
\[
\frac{\acd_p(G)-1}{f}[G:I_j]\geq [G:I_j]-\acd_p(G).
\]
This is equivalent to
\[
\frac{[G:I_j](f+1)}{[G:I_j]+f}\leq \acd_p(G),
\]
and the proof is complete.
\end{proof}

Finally we can prove the main theorem.

\begin{theorem}\label{theorem-main-1-again}
Let $p$ be a prime and $G$ a finite group. Let $\ell(p)$ be the smallest positive integer such that
$\ell(p)p+1$ is a prime power. If
$$\acd_p(G)< b_p:= \frac{2\ell(p)p}{\ell(p)p+1},$$
then $G$ has a normal Sylow $p$-subgroup.
\end{theorem}

\begin{proof}
Since the theorem has been proved for $p=2$ in
\cite[Theorem~1.1]{Hung-Tiep}, we assume from now on that $p$ is
odd. We will argue by induction on $|G|$. Let $G$ be a finite group
with $\acd_p(G)<b_p$. By Theorem~\ref{theorem-main-2-again-2},
we know that $\bO^{p'}(G)$ is solvable.

First we consider the case $G'\cap \bO^{p'}(G)$ is trivial. Then
$\bO^{p'}(G)$ can be viewed as a subgroup of $G/G'$, and thus it is
abelian. It follows that the Sylow $p$-subgroup of $\bO^{p'}(G)$ is
normal in $G$, and we are done.

From now on we will assume that $G'\cap \bO^{p'}(G)$ is nontrivial.
In particular, we can choose a minimal normal subgroup $N$ of $G$
that is inside both $G'$ and $\bO^{p'}(G)$. Observe that $N \cong C_r^m$ is
elementary abelian since $\bO^{p'}(G)$ is solvable. Furthermore,
$\acd_p(G/N)\leq\acd_p(G)<b_p$ by Lemma \ref{L-subset}.
The induction hypothesis now implies that $G/N$ has a
normal Sylow $p$-subgroup, say $P_1/N$.
If $N$ is a $p$-group then $P_1$ is a normal Sylow $p$-subgroup of
$G$, and we are done.

So we will assume that $N$ is an
elementary abelian $p'$-group, and so
$P_1=N\rtimes P$ for a Sylow
$p$-subgroup of $P_1$. By Frattini's argument we have
$G=P_1\bN_G(P)=N\bN_G(P)$. If $N$ is contained in the Frattini
subgroup $\Phi(G)$ of $G$, we would have $G=\bN_G(P)$, which means that $P
\unlhd G$, and we are done. So it remains to consider the case that
$N \not\leq \Phi(G)$. We then choose
a maximal subgroup $H$ of $G$ such that $N\nleq H$, so that
$G=NH$. Now $N\cap H \unlhd H$, and $N \cap H \unlhd N$ as $N$ is abelian. It follows that
$N\cap H\unlhd G$, and thus $N\cap H=1$ by the minimality of $N$. We conclude
that $G$ is a split extension of $N$ by $H$.

Recall $P_1 = N \rtimes P \unlhd G$. Suppose we can find a non-principal irreducible character $\lambda$ of $N$ that is $P_1$-invariant.
Then, for every $g\in P_1$
and $n\in N$, we have
\[
\lambda(n)=\lambda^g(n)=\lambda(gng^{-1}),
\]
and thus $\lambda([g,n])=1$. Since $\lambda$ is not the principal
character, we deduce that $[N,P_1] \leq \Ker(\lambda)$ is a proper subgroup of $N$. The
minimality of $N$ then implies that $[N,P_1]=1$ since $[N,P_1]\lhd G$.
This means that $P_1=N\times P$, and
we can again conclude that $P\unlhd G$.

We may now assume that every $P_1$-orbit, whence every $H$-orbit, on
$\Irr(N) \smallsetminus \{1_N\}$, has size divisible by $p$.
We now can apply Lemma~\ref{lemma-orbit} to deduce that, in the
action of $H$ on $\Irr(N)\smallsetminus \{1_N\}$, there is an orbit
$\mathcal{O}$ such that $p\mid |\mathcal{O}|$
and
\begin{equation}\label{est2}
  \frac{|\mathcal{O}|(f+1)}{|\mathcal{O}|+f}\leq \acd_p(G),
\end{equation}
where $f$ is the number of $H$-orbits on $\Irr(N) \smallsetminus \{1_N\}$.
If furthermore $f \geq 2$, then
$$\frac{|\OC|(f+1)}{|\OC|+f} \geq \frac{3p}{p+2} > b_p > \acd_p(G),$$
contradicting \eqref{est2}. Thus $f=1$, and so $|\OC| = |\Irr(N) \smallsetminus \{1_N\}| = r^m-1$.
By the definition of $\ell(p)$, we have $|\OC| \geq \ell(p)p$, and so
$$\frac{|\OC|(f+1)}{|\OC|+f} = \frac{2\ell(p)p}{\ell(p)p+1} = b_p > \acd_p(G),$$
again contradicting \eqref{est2}.
\end{proof}


\section{Proof of Theorems~\ref{theorem-simple-groups-key1} and
\ref{theorem-simple-groups-key2}}\label{section-last}

Note that Theorem \ref{theorem-simple-groups-key1} has been
established for $p=2$ in \cite{Hung-Tiep}, and $p$ must be odd in
Theorem \ref{theorem-simple-groups-key2}. We therefore assume that
$p> 2$ for the rest of the paper.

We first observe the following, which handles Theorem
\ref{theorem-simple-groups-key1} in several cases.

\begin{lemma}\label{proposition-Out-cyclic}
Theorem \ref{theorem-simple-groups-key1} holds if $\Out(S)$ is
cyclic.
\end{lemma}

\begin{proof}
By the Ito-Michler theorem, for every prime divisor $p$ of $|S|$,
there exists an irreducible character $\theta$ of $S$ with $p\mid
\theta(1)$. Since $I_{\Aut(S)}(\theta)/S$ is a subgroup of
$\Out(S)$, out hypothesis on $\Out(S)$ implies that
$I_{\Aut(S)}(\theta)/S$ is cyclic. \cite[Corollary 11.22]{Isaacs1}
in turns then implies that $\theta$ is extendible to
$I_{\Aut(S)}(\theta)$, as desired.
\end{proof}

Lemma \ref{proposition-Out-cyclic} implies
Theorem~\ref{theorem-simple-groups-key1} (at least) for the
alternating groups $\Al_n$ with $n\geq 7$ and all sporadic simple
groups, since their outer automorphism groups are
trivial or of order 2. For the same reason,
Theorem~\ref{theorem-simple-groups-key2} also holds for these simple groups.
Therefore we may henceforth assume that $S$ is a simple
group of Lie type, say in characteristic $r$. We will assume furthermore that $S\neq \Sp_4(2)'\cong \Al_6, \ta F_4(2)'$
as these cases can be confirmed easily using \cite{Atl1}.

\smallskip
To prove Theorem \ref{theorem-simple-groups-key1}, we will produce a
set of irreducible characters of $S$, each of which is extendible to
its inertia subgroup in $\Aut(S)$, such that the product of their
degrees is divisible by every prime divisor of $|S|$. According to
\cite{Feit}, the Steinberg character $\St$ of degree $|S|_r$ of $S$
extends to $\Aut(S)$. Therefore, we
only need to produce such a set of irreducible characters
of $S$ such that the product of their degrees is divisible by every
prime divisor of $|S|_{r'}$.

We remark that many characters in the required set that we are about
to construct are in fact real-valued or even strongly real (that is, of Frobenius-Schur indicator $1$).
We hope that this property will be useful in other applications.

\subsection{Theorem \ref{theorem-simple-groups-key1} for classical groups in odd
characteristic}\label{SS-odd}
Let $\mathcal{G}$ be a simple algebraic group of adjoint type
defined over a field of characteristic $r$ and
$F:\mathcal{G}\rightarrow \mathcal{G}$ a Steinberg endomorphism such
that $S=[G,G]$ for $G:=\mathcal{G}^F$. Let $(\GC^*,F^*)$ be
dual to $(\mathcal{G},F)$ and set $G^*:=\GC^{*F^*}$. According
to the Deligne-Lusztig theory \cite{Carter,Digne-Michel} on complex
representations of finite groups of Lie type, the set of irreducible
complex characters of $G$ is partitioned into rational series
$\mathcal{E}(G,(s))$, which are labeled by conjugacy classes $(s)$
of semisimple elements of $G^*$.
For any semisimple element $s\in G^*$, there is a bijection $\chi\mapsto\psi$ from $\mathcal{E}(G,(s))$
to $\mathcal{E}(\bC_{G^*}(s),(1))$ such that
\[
\chi(1)=[G^*:\bC_{G^*}(s)]_{r\p}\psi(1).
\]
Since $\GC^*$ is simply connected, $\bC_{\GC^*}(s)$ is connected, and
the character in $\mathcal{E}(G,(s))$ corresponding to the trivial
character of $\mathcal{E}(\bC_{G^*}(s),(1))$, denoted by $\chi_s$, is
called the semisimple character associated to $s$. The characters in
$\mathcal{E}(G,(1))$ are unipotent characters of $G$.
It is well known that the unipotent characters restricts irreducibly to $S$ and
thus we also call them unipotent characters of $S$. Unipotent
characters of finite groups of Lie type have been completely
described in \cite[Sections 13.8 and 13.9]{Carter} and we will use
these descriptions without further notice.


We will frequently use the following property of unipotent characters:

\begin{lemma}\label{lemma Malle 2} {\rm \cite[Theorem~2.4]{Malle}}
Let $S$ be a simple group of Lie type. Then every unipotent
character of $S$ extends to its inertia subgroup in $\Aut(S)$.
\end{lemma}


Let us also recall the following fact, which was established in \cite{Tiep}.

\begin{lemma}\label{lemma-TiepProp-5.1}
In the above notation, let $s$ be a real semisimple element of order
relatively prime to $\bZ(G^*)$. Then $\chi_s$ is real-valued and
restricts irreducibly to $S$. Furthermore, if $\chi_s(1)$ is odd,
then $\theta:=(\chi_s)_S$ is extendible to a (strongly real)
character of $I_{\Aut(S)}(\theta)$.
\end{lemma}

\begin{proof}
See \cite[Lemma 2.2 and Proposition 5.1]{Tiep}.
\end{proof}

\subsubsection{$S=\PSL_n(q)$}

A) First we consider $\PSL_2(q)$ with $q\geq 5$. As the alternating groups of degree $5$ and $6$ were already handled,
we assume that $q\neq 5,9$. Irreducible characters
of two-dimensional linear groups are well known, see \cite{White}
for instance. According to \cite[p.~8]{White}, when $q$ is odd,
irreducible characters of degrees $q\pm1$ of $\PSL_2(q)$ can be
labeled by:

\begin{enumerate}
\item[(i)] $\chi_i, 1\leq i\leq (q-3)/2$ and $i$ even, of degree $q+1$,
\item[(ii)] $\theta_j, 1\leq j\leq (q-1)/2$ and $j$ even, of degree
$q-1$,
\end{enumerate}
Here, for the reader's convenience, we use the same notation as in
\cite{White}. Let $\varphi$ be the field automorphism of order $f$
of $\PSL_2(q)$. Then, by \cite[Lemma~4.8]{White}, the character
$\chi_i\in\Irr(\PSL_2(q))$ is invariant under $\varphi^k$ where
$1\leq k\leq f$ if and only if $(p^f-1)\mid i(p^k-1)$ or
$(p^f-1)\mid i(p^k+1)$; and the character
$\theta_j\in\Irr(\PSL_2(q))$ is invariant under $\varphi^k$ if and
only if $(p^f+1)\mid j(p^k-1)$ or $(p^f+1)\mid j(p^k+1)$. It is then
routine to check that, when $q\neq 5,9$, the characters $\chi_2$ and
$\theta_2$ are not invariant under any field automorphism. It is
well known that every irreducible character of $\PSL_2(q)$ of degree
$q\pm 1$ is invariant under the diagonal automorphism. Thus,
\[
I_{\Aut(S)}(\chi_2)=I_{\Aut(S)}(\theta_2)=\PGL_2(q).
\]
As $\PGL_2(q)/S\cong C_2$, both $\chi_2$ and $\theta_2$ are
extendible to $\PGL_2(q)$ and we now have a required set
$\{\St_S,\chi_2, \theta_2\}$.

\medskip

B) Now we can suppose $n\geq 3$. Case IIa of the proof of
\cite[Proposition 4.7]{Marinelli-Tiep} produces a regular semisimple
element $s_1\in G^*$ such that the semisimple character $\chi_{s_1}\in
\Irr(G)$ satisfies $\theta_1=(\chi_{s_1})_S\in \Irr(S)$ and
$I_{\Aut(S)}(\theta_1)=G$. Moreover,
\[|\bC_{G^*}(s)|=(q^m-1)(q-1)^{n-m-1},\] where $m\in \{n,n-1\}$ is
chosen to be odd. We then have
\[\theta_1(1)=[G^*:\bC_{G^*}(s_1)]_{r'}=\frac{\prod_{i=2}^{n}(q^i-1)}{(q^m-1)(q-1)^{n-m-1}}.\]

Suppose that $n=3$. Consider the semisimple character
$\chi_{s_2}\in\Irr(\PGL_3(q))$ where $s_2\in G^*=\SL_3(q)$ is a
diagonal matrix with eigenvalues $-1,-1,1$. As $s_2$ is real
of order $2$ and $|\bZ(G^*)|=\gcd(3,q-1)$,
Lemma~\ref{lemma-TiepProp-5.1} implies that
$\theta_2:=(\chi_{s_2})_S$ is irreducible. Moreover,
\[
\theta_2(1)=[G^*:\bC_{G^*}(s_2)]_{r'}=\frac{(q^2-1)(q^3-1)}{(q-1)(q^2-1)}=q^2+q+1,
\]
which is odd. Lemma~\ref{lemma-TiepProp-5.1} again yields that
$\theta_2$ is extendible to $I_{\Aut(S)}(\theta_2)$.

Now suppose that $n\geq 4$. Then part (c) of the proof of \cite[Proposition 5.5]{NT1} yields a character
$\theta_2 \in \Irr(S)$ of degree
\[\theta_2(1)=\left\{\begin {array}{ll}
(q^n-1)(q^{n-1}-1)/(q-1)^2, & \text{if } q \equiv 1 (\bmod 4),\\
(q^n-1)(q^{n-1}-1)/(q^2-1), & \text{if } q \equiv 3 (\bmod 4),
 \end {array} \right.\]
such that $\theta_2$ is extendible to $\Aut(S)$. The set
$\{\St_S,\theta_1,\theta_2\}$ now satisfies our requirement for all
$n\geq 3$.

\subsubsection{$S=\PSU_n(q)$ with $n\geq 3$}
It was shown in part IIb of the proof of \cite[Proposition
4.7]{Marinelli-Tiep} that there is a regular semisimple element
$s_1\in G^*$ so that the semisimple character $\chi_{s_1}\in \Irr(G)$
satisfies the conditions $\theta_1=(\chi_{s_1})_S\in \Irr(S)$ and
$I_{\Aut(S)}(\theta_1)=G$. Moreover,
\[|\bC_{G^*}(s_1)|=(q^m+1)(q+1)^{n-m-1},\] where $m\in \{n,n-1\}$ is
chosen to be odd. We then have
\[\theta_1(1)=[G^*:\bC_{G^*}(s_1)]_{r'}=\frac{\prod_{i=2}^{n}(q^i-(-1)^i)}{(q^m+1)(q+1)^{n-m-1}}.\]

When $n=3$, a similar construction as in the $\PSL_3(q)$ case yields
a character $\theta_2\in\Irr(S)$ of degree $q^2-q+1$ and $\theta_2$
extends to $I_{\Aut(S)}(\theta_2)$, and we obtain the required set
$\{\St_S,\theta_1,\theta_2\}$.

So we can assume that $n\geq 4$. According to the proof of
\cite[Theorem 2.1]{Dolfi-Navarro-Tiep}, $\Aut(S)$ has a rank $3$
permutation character $\rho=1+\alpha+\beta$ such that $\alpha_S$ and
$\beta_S$ are both irreducible. Moreover,
\[\alpha(1)=\frac{q^2(q^n-(-1)^n)(q^{n-3}-(-1)^{n-3})}{(q+1)(q^2-1)}\]
and
\[\beta(1)=\frac{q^3(q^{n-1}-(-1)^{n-1})(q^{n-2}-(-1)^{n-2})}{(q+1)(q^2-1)}.\]
We now have the set $\{\theta_1,\alpha_S,\beta_S\}$ with desired properties.

\subsubsection{$S=\PSp_{2n}(q),~\Omega_{2n+1}(q)$ with $n\geq 2$.}

It was shown in the proof of \cite[Proposition 4.5]{Marinelli-Tiep}
that if $s$ is a semisimple simple element of $G^*$ of order a
{\it primitive prime divisor} of $r^{2nf}-1$ (see \cite{Zs} for the definition and
existence of such divisors), then the semisimple
character $\chi_s\in\Irr(G)$ restricts irreducibly to $S$ and
$I_{\Aut(S)}(\theta)=G$ with $\theta:=(\chi_s)_S$. Furthermore,
$|\bC_{G^*}(s)|=q^n+1$, which implies that
\[\theta(1)=[G^*:\bC_{G^*}(s)]_{r'}=(q^n-1)\prod_{i=1}^{n-1}(q^{2i}-1).\]

On the other hand, $S$ has a unipotent character $\chi$,
parametrized by the symbol $1\hspace{3pt} n \choose 0$, of degree

\[\chi(1)=\frac{q(q^n+1)(q^{n-1}-1)}{2(q-1)},\] see \cite[Corollary 3.2]{Nguyen}.
Lemma \ref{lemma Malle 2} asserts that $\chi$ is extendible to
$I_{\Aut(S)}(\chi)$ (which is in fact $\Aut(S)$ in this case). Now
the set $\{\theta,\chi\}$ fulfills our requirements.

\subsubsection{$S=P\Omega_{2n}^\pm(q)$ with $n\geq 4$ even.}

The case $S=P\Omega_{2n}^-(q)$ with $2|n$ can be handled
similarly as in the symplectic case with the note that the unipotent
character $\chi$ is parametrized by the symbol $1\hspace{3pt} n-1
\choose -$ and has degree
\[\chi(1)=\frac{q(q^n+1)(q^{n-2}-1)}{q^2-1},\] see \cite[Proposition 3.3]{Nguyen}.

Now we consider $S=P\Omega_{2n}^+(q)$ with $2|n$. When
$n\geq 6$, part 6.5 of the proof of \cite[Theorem 3.1]{Tiep}
produces characters $\theta_{1,2}\in\Irr(S)$ such that
$I_{\Aut(S)}(\theta_{1,2})=G$ and $\theta_{1,2}$ are extendible to
$G$. Furthermore
\[\theta_1(1)=\frac{(q^n-1)\prod_{i=1}^{n-1}(q^{2i}-1)}{(q^2+1)(q^{n-2}+1)},\]
and
\[\theta_2(1)=\left\{\begin {array}{ll}
\dfrac{(q^n-1)\prod_{i=1}^{n-1}(q^{2i}-1)}{(q+\varepsilon)(q^{n-1}+\varepsilon)}, & \text{if } q\geq 5,\\
\dfrac{(q^n-1)\prod_{i=1}^{n-1}(q^{2i}-1)}{26(3^{n-3}-1)}, &
\text{if } q =3,
 \end {array} \right.\]
where $q\equiv \varepsilon (\bmod 4)$. The set
$\{\St_S,\theta_1,\theta_2\}$ satisfies our requirements.

Finally we consider $S=P\Omega_8^+(q)$. This group has three
unipotent characters, parametrized by the symbols $1 \choose 3$,
$0\hspace{3pt}1\hspace{3pt} 2\hspace{3pt} \choose -$, and
$0\hspace{3pt} 2 \choose 1\hspace{3pt}3$ of degrees
\[q(q^2+1)^2, \frac{1}{2}q^3(q-1)^4(q^2+q+1), \text{ and }
\frac{1}{2}q^3(q+1)^4(q^2-q+1),
\]
respectively. These three characters form our desired set.

\subsubsection{$S=P\Omega_{2n}^\varepsilon(q)$ with $n\geq 5$ odd.}

As shown in part 6.7 of the proof of \cite[Theorem 3.1]{Tiep}, there
exists a character $\theta\in\Irr(S)$ which extends to a semisimple
character of $G$ of degree
\[\theta(1)=\Pi_{i=1}^{n-1}(q^{2i}-1).\] Moreover $I_{\Aut(S)}(\theta)=G$.

We have mentioned in the previous subsection that
$S=P\Omega_{2n}^-(q)$ has a unipotent character $\chi$ parametrized
by the symbol $1\hspace{3pt} n-1 \choose -$ and has degree
\[\chi(1)=\frac{q(q^n+1)(q^{n-2}-1)}{q^2-1}.\] On the other hand, $S=P\Omega_{2n}^+(q)$
has a unipotent character $\chi$ parametrized by the symbol $n-1
\choose 1$ of degree
\[\chi(1)=\frac{q(q^n-1)(q^{n-2}+1)}{q^2-1},\] see \cite[Proposition
3.4]{Nguyen}. We now have the set $\{\theta,\chi\}$ as required.

\subsection{Theorem \ref{theorem-simple-groups-key1} for classical groups in even characteristic}
In this subsection, $S$ is a simple classical group defined over a field $\FF_q$ in
characteristic $2$, with $q=2^f$. Departing from the viewpoint of $S$ as $[G,G]$ used in
\S\ref{SS-odd}, here we can find a simple simply connected
algebraic group $\HC$ and a Steinberg endomorphism $F:\mathcal{H}\rightarrow \mathcal{H}$
such that $S=H/\bZ(H)$ for $H:=\mathcal{H}^F$.
Let $(\mathcal{H}^*,F^*)$ be dual to $(\mathcal{H},F)$, and let
$H^*:=\mathcal{H}^{*F^*}$. The following lemma is a part of
\cite[Proposition 7.1]{Tiep}.

\begin{lemma}\label{lemma-even-characteristic}
In the above notation, assume that $s\in [H^*,H^*]$ is a real semisimple
element such that $\bC_{\GC^*}(s)$ is connected. Then the
semisimple character $\chi_s\in\Irr(H)$ is trivial at $\bZ(H)$ and
hence it can be viewed as a character of $S$. Moreover, it is
extendible to a (strongly real) character of its inertia subgroup in
$\Aut(S)$.
\end{lemma}


Note that $\Out(S)$ can be read off from \cite[Theorem 2.5.12]{Gorenstein-Lyons-Solomon}.
Due to Lemma \ref{proposition-Out-cyclic}, we only need to
consider the groups $\PSL_n(q)$ with $n\geq 3$, $\PSU_n(q)$ with
$n\geq 3$, $\Sp_4(q)$, and $\Omega_{2n}^+(q)$.

\subsubsection{$S=\PSL_n(q)$ with $n\geq 3$}

First we suppose $n=3$. As the case $S=\PSL_3(8)$ can be checked
easily using \cite{Atl1}, we assume that $2 < q = 2^f\neq 8$. By \cite{Zs}, we can then find a primitive prime
divisor $\ell > 3$ of $2^{2f}-1$. Choose a real semisimple element
$s\in\PSL_2(q)<[H^*,H^*]$ of order $\ell$ and consider the associated
semisimple character $\chi_s\in\Irr(\SL_3(q))$ of degree
\[
\chi_s(1)=[H^*:\bC_{H^*}(s)]_{2'}=q^3-1.
\]
By Lemma \ref{lemma-even-characteristic}, $\chi_s$ can be
viewed as a character of $S$ and it is extendible to its inertia
subgroup in $\Aut(S)$. We also note that $\PSL_3(q)$ has a unipotent
character $\chi^{(1,2)}$ of degree $q(q+1)$. The set
$\{\chi_s,\chi^{(1,2)}\}$ fulfills the requirements.

When $S=\PSL_6(2)$, we simply take the set of three unipotent
characters parametrized by the partitions $(1,5)$, $(2,4)$, and
$(1,2,3)$ of degrees $62, 588$, and $6480$, respectively. When
$S=\PSL_7(2)$, we take the set of three unipotent characters
parametrized by the partitions $(1,6)$, $(2,5)$, and $(1,1,5)$ of
degrees $126$, $2540$, and $5208$, respectively.

Now we may assume that $n\geq 4$ and $(n,q)\neq (6,2), (7,2)$. Choose
$m\in\{n-1,n\}$ to be even. The assumption of $n$ and $q$ implies by \cite{Zs}
that $2^{mf}-1$ has a primitive prime divisor, say $\ell$. Following
part 7.3 of the proof of \cite[Theorem 3.1]{Tiep}, we choose an
element $s\in\PSL_n(q)=[H^*,H^*]$ with a preimage of order $\ell$ in
$\Sp_m(q)<\SL_n(q)=L$. This element $s$ is real and
$\bC_\mathcal{G}(s)$ is connected, and therefore $\chi_s$, viewed as
a character of $S$, extends to its inertia subgroup in $\Aut(S)$.
Furthermore,
\[
\chi_s(1)=[H^*:\bC_{H^*}(s)]_{r'}=\frac{\prod_{i=2}^n(q^{i}-1)}{(q-1)^{n-m-1}(q^m-1)}.
\]
The required set of characters will be $\{\chi_s,
\chi^{(1,n-1)},\chi^{(2,n-2)}\}$, where $\chi^\alpha$ denotes the
unipotent character of $S$ parametrized by partition $\alpha \vdash n$. Note
that
\[
\chi^{(1,n-1)}(1)=\frac{q(q^{n-1}-1)}{q-1} \text{ and }
\chi^{(2,n-2)}(1)=\frac{q^2(q^{n}-1)(q^{n-3}-1)}{(q-1)(q^2-1)}.
\]

\subsubsection{$S=\PSU_n(q)$ with $n\geq 3$}

The case $S=\PSU_3(q)$ can be argued similarly as in the $\PSL_3(q)$
case: here we can find two irreducible characters of degrees $q^3+1$ and
$q(q-1)$ satisfying our conditions. So we suppose that $n\geq 4$.
Again as in the linear case, one can construct a real semisimple
element $s\in [H^*,H^*]=\PSU_n(q)$ with an inverse image in
$\Sp_m(q)<\SU_n(q)$ of order $\ell$, where $m\in\{n,n-1\}$ is even
and $\ell$ is a primitive prime divisor of $2^{mf}-1$ if $4\mid m$ and
a primitive prime divisor of $2^{mf/2}-1$ if $4\nmid m$. Then, as shown in part 7.3 of
the proof of \cite[Theorem 3.1]{Tiep}, the
semisimple character $\chi_s\in\Irr(S)$ extends to its inertia
subgroup in $\Aut(S)$, and
\[
\chi_s(1)=[H^*:\bC_{H^*}(s)]_{r'}=\frac{\prod_{i=2}^n(q^{i}-(-1)^i)}{(q+1)^{n-m-1}(q^m-1)}.
\]
This $\chi_s$ together with the unipotent characters
$\chi^{(1,n-1)},\chi^{(2,n-2)}$ of degrees
\[
\chi^{(1,n-1)}(1)=\frac{q(q^{n-1}-(-1)^{n-1})}{q+1} \text{ and }
\chi^{(2,n-2)}(1)=\frac{q^2(q^{n}-(-1)^n)(q^{n-3}-(-1)^{n-3})}{(q+1)(q^2-1)}.
\]
will form a required set.

\subsubsection{$S=\Sp_4(q)$ with $q\geq 4$} This group has three
unipotent characters parametrized by the symbols $0\hspace{3pt} 1
 \hspace{3pt}2\choose-$, $0\hspace{3pt} 2 \choose 1$, and $0\hspace{3pt} 1 \choose
 2$ of degrees
 \[
 \frac{1}{2}q(q-1)^2, \frac{1}{2}q(q+1)^2, \text{ and }
 \frac{1}{2}q(q^2+1),
 \]
respectively. It is easy to check that every prime divisor of $|S|$
divides at least one of these degrees. Note that $\Sp_4(2)' \cong \AAA_6$ was already considered before.

\subsubsection{$S=\Omega_{2n}^+(q)$ with $n\geq 4$}

As the case $S=\Omega_8^+(2)$ can be checked directly using
\cite{Atl1}, we assume that $(n,q)\neq (4,2)$. We then have $H\cong
S\cong H^*$. Let $\ell$ be a primitive prime divisor of $2^{(2n-2)f}-1$
and choose $s$ to be a real semisimple element of $G$ of order
$\ell$. This can be done since all semisimple elements of
$\Omega_{2n}^{\pm}(q)$ are real when $n$ is even, by \cite[Proposition
3.1]{Tiep-Zalesski}. When $n$ is odd, we just choose
$s_1\in\Omega_{2n-2}^-(q)<\Omega_{2n}^+(q)$. We now have a
semisimple character of $S$ of degree
\[
\chi_{s}(1)=[H^*:\bC_{H^*}(s)]_{r'}=\frac{(q^n-1)\prod_{i=1}^{n-1}(q^{2i}-1)}{(q+1)(q^{n-1}+1)},
\]
which is extendible to its inertia subgroup in $\Aut(S)$, by
Lemma~\ref{lemma-even-characteristic}. On the other hand, $S$ has a
unipotent character parametrized by
$0\hspace{3pt}1\choose1\hspace{3pt}n$ of degree
\[
\frac{q^{2n}-q^2}{q^2-1},
\]
see \cite[Proposition 3.4]{Nguyen}. This unipotent character and
$\chi_s$ above will satisfy our conditions.


\subsection{Theorem \ref{theorem-simple-groups-key1} for exceptional groups of Lie type}
By Lemma \ref{proposition-Out-cyclic}, we only need to consider
families with possible non-cyclic outer automorphism group, which
are $G_2(3^f)$, $F_4(2^f)$, $E_6(q)$, and $E_7(q)$. However, our
arguments below can be applied to all exceptional groups of Lie
type.

Unipotent characters of groups of exceptional type are well known.
It turns out that there are always unipotent characters of $S$ such
that the product of their degrees is divisible by $|S|$. To write
down their degrees simply, let us denote $\Phi_k$ the
$k^{\mathrm {th}}$ cyclotomic polynomial evaluated at $q$, and use the notation
in \cite[Section 13.9]{Carter}.

For $S=G_2(q)$ with $q=3^f$, the characters $\phi_{1,3'}$,
$G_2(\theta)$, and $\phi_{1,6}$ of degrees
$$\frac{1}{3}q\Phi_3\Phi_6, ~\frac{1}{3}\Phi_1^2\Phi_2^2,~q^6,$$
respectively, satisfy our requirement. For $S=F_4(q)$ with
$q=2^f$, the characters $\phi_{8,3'}$, $F_4[i]$, and $\phi_{1,24}$
of degrees
$$q^3\Phi_4^2\Phi_8\Phi_{12},~\frac{1}{4}q^4\Phi_1^4\Phi_2^4\Phi_3^2\Phi_6^2,~q^{24},$$
respectively, will do the job. For $S=E_6(q)$, we choose the
characters $\phi_{81,6}$, $\phi_{1,36}$, $E_6[\theta]$, and
$E_6[\theta^2]$ of degrees
$$q^6\Phi_3^3\Phi_6^2\Phi_9\Phi_{12},~q^{36},~\frac{1}{3}q^7\Phi_1^6\Phi_2^4\Phi_4^2\Phi_5\Phi_8,~
   \frac{1}{3}q^7\Phi_1^6\Phi_2^4\Phi_4^2\Phi_5\Phi_8,$$
respectively.
Finally, for $S=E_7(q)$, we choose the characters $\phi_{27,2}$,
$\phi_{189,5}$, $\phi_{1,63}$, $E_6[\theta]$, and $E_6[\theta^2]$ of
degrees
$$\begin{aligned}q^2\Phi_3^2\Phi_6^2\Phi_9\Phi_{12}\Phi_{18},~q^5\Phi_3^2\Phi_6^2\Phi_7\Phi_9\Phi_{12}\Phi_{14}\Phi_{18},
~q^{63},\\
\frac{1}{3}q^7\Phi_1^6\Phi_2^6\Phi_4^2\Phi_5\Phi_7\Phi_8\Phi_{10}\Phi_{14},
~\frac{1}{3}q^7\Phi_1^6\Phi_2^6\Phi_4^2\Phi_5\Phi_7\Phi_8\Phi_{10}\Phi_{14},\end{aligned}$$
respectively.

Theorem \ref{theorem-simple-groups-key1} is now completely proved.


\subsection{Proof of Theorem \ref{theorem-simple-groups-key2}}\label{SS-even}

Finally we prove Theorem \ref{theorem-simple-groups-key2}, which we restate:

\begin{theorem}
Let $p$ be a prime and let $S$ be a non-abelian simple group of
order not divisible by $p$. Then there exists a non-principal
character $\theta\in\Irr(S)$ such that $\theta$ is extendible to a
character of $I_{\Aut(S)}(\theta)$ and $p\nmid
|I_{\Aut(S)}(\theta)|$.
\end{theorem}

\begin{proof}
We already mentioned above that the theorem is obvious for the alternating
groups and sporadic simple groups. Hence we will assume that $S$ is a simple
group of Lie type, defined over of field of
$q=r^f$ elements.

First we suppose that $q$ is odd. As in \S\ref{SS-odd}, we have $S=[G,G]$ where
$G=\mathcal{G}^F$ with $\mathcal{G}$ a simple algebraic group of
adjoint type defined over a field of characteristic $r$ and
$F:\mathcal{G}\rightarrow \mathcal{G}$ a Steinberg endomorphism. It has
already been shown in Subsection~\ref{SS-odd} that,
when $G$ is a classical group in odd characteristic, it possesses a
semisimple character $\chi_s$ where $s$ is a semisimple element of
the dual group $G^*$ such that $\theta:=(\chi_s)_S\in\Irr(S)$ and
$I_{\Aut(S)}(\theta)=G$. Indeed, the proofs of \cite[Propositions
4.4, 4.5, and 4.7]{Marinelli-Tiep} produced such a character for all
simple groups of Lie type in odd characteristic. So we are done when
$q$ is odd.

Now suppose that $q$ is even. Recall that the order of $\Out(S)$ is
$dfg$, where $d$ is the order of the group of diagonal
automorphisms, $f$ is the order of the cyclic group of field
automorphisms (generated by a Frobenius automorphism), and $g$ is
the order of the group of graph automorphisms coming from
automorphisms of the Dynkin diagram, see \cite[Theorem
2.5.12]{Gorenstein-Lyons-Solomon}. Furthermore, $|S|$ is always
divisible by $d$ and $g$. If $p \nmid f$, then we can consider any non-principal character
$\theta \in \Irr(S)$ that extends to $I_{\Aut(S)}(\theta)$ constructed in the proof of Theorem
\ref{theorem-simple-groups-key1} and observe that $p \nmid |I_{\Aut(S)}(\theta)|$.

So we may assume that $p\mid f$.
The proof of \cite[Proposition 5.8]{Tiep} constructed a (real)
semisimple element $s\in G^*$ of order coprime to $|\bZ(G^*)|$ such that
the $G^*$-conjugacy class of $s$ is not invariant under $\rho^{f/p}$,
where $\rho$ is a generator of the cyclic group of field
automorphisms. By Lemma~\ref{lemma-TiepProp-5.1} and \cite[Proposition 5.1(iii)]{Tiep}, the semisimple
character $\chi_s$ (of degree $[G^*:\bC_{G^*}(s)]_{2'}$, which is odd) then restricts
irreducibly to $S$, and $\theta := (\chi_s)_S$ extends to a (strongly real)
character of $I_{\Aut(S)}(\theta) = I_{\Aut(S)}(\chi_s)$. Moreover, as the
$G^*$-conjugacy class of $s$ is not invariant under $\rho^{f/p}$, we
have $p\nmid [I_{\Aut(S)}(\chi_s):S]$, and so the character
$\theta$ fulfills our requirements.
\end{proof}

\subsection{Final remarks} By inspecting the unipotent characters and their degrees
of Lie-type groups of low rank, it seems to us most of the prime
divisors of such a group $S$ divide the degree of a unipotent
character of $S$. For instance, if $S\neq \ta B_2(q)$ is a simple
group of exceptional type, every prime divisor of $|S|$ divides the
degree of a unipotent character of $S$, see \cite[\S 13.9]{Carter}.

For classical groups, the same assertion is not true for linear
groups, unitary groups, and non-split orthogonal groups in even
dimension. However, it looks plausible that, when $S$ is a symplectic group,
an orthogonal group in odd dimension, or a split orthogonal group
in even dimension, every prime divisor of $|S|$ divides the
degree of a unipotent character of $S$. It would be useful to
confirm this phenomenon, as it helps to conveniently establish
results on divisibility and extendibility of characters, like
Theorems~\ref{theorem-simple-groups-key1} and
\ref{theorem-simple-groups-key2}.



\end{document}